\documentclass[10pt]{amsart}
\usepackage{amssymb,amsmath,amsthm,amsfonts,amsopn,url}
\usepackage{amscd,amssymb,amsopn,amsmath,amsthm,graphics,amsfonts,enumerate,verbatim,calc}
\usepackage[mathcal]{euscript}
\usepackage[all]{xy}
\theoremstyle{plain}
\newtheorem{thm}{Theorem}[section]
\newtheorem*{mt*}{Main Theorem}

\newtheorem{prop}[thm]{Proposition}
\newtheorem{lemma}[thm]{Lemma}
\newtheorem{cor}{Corollary}
\newtheorem{rem}{Remark}
\newtheorem{example}{Example}

\theoremstyle{definition}

\newcommand{\ideal}[1]{\mathfrak{#1}}
\newcommand{\m}{\ideal{m}}

\newcommand{\p}{\ideal{p}}
\newcommand{\q}{\ideal{q}}

\newcommand{\func}[1]{\mathrm{#1} \,}

\newcommand{\Ass}{\func{Ass}}
\newcommand{\Ann}{\func{Ann}}
\newcommand{\Supp}{\func{Supp}}
\newcommand{\Spec}{\func{Spec}}
\newcommand{\rad}{\func{rad}}
\newcommand{\V}{\func{V}}
\newcommand{\T}{\func{T}}
\newcommand{\Z}{\func{Z}}

\newcommand{\ZZ}{{\mathbb Z}}
\title[]{Behaviour of the Support of Lyubeznik Functors under Ring Extensions}

\author[]{Rajsekhar Bhattacharyya}

\address{Dinabandhu Andrews College,\newline 54 Raja S.C. Mallick Road, Garia,\newline Kolkata 700084,\newline India}

\email{rbhattacharyya@gmail.com}

\keywords{Local Cohomology}

\subjclass[2010]{13D45}
\begin{document}

\begin{abstract}
Let $R\rightarrow S$ be an arbitrary ring extension of Noetherian rings. In this article we study the behaviour of Zariski closedness of the support of Lyubeznik functors $\T$, when the ring extension $R\rightarrow S$ is namely `flat', `faithfully flat', `pure' and lastly `cyclically pure'. We show that the Zariski closedness of the support comes down from extended ring to the base ring for faithfully flat, pure and finally for cyclically pure ring extensions. Lastly, we focus on a special case of pure extension i.e. when $R$ is a direct summand of $S$ and we compare the sets $\Supp_S(\T(R)\otimes_R S)$ and $\Supp_S \T(S)$. 
\end{abstract}

\maketitle

\section{introduction}

Let $R$ be a Noetherian ring and let $M$ be an arbitrary $R$-module. For an ideal $I\subset R$, and for an integer $i\geq 0$, $i$-th local cohomology module with support in the ideal $I$ is denoted by $H^i_I (M)$. In general, local cohomology modules are not finitely generated modules. In \cite{Hu}, it was asked that whether local cohomology modules of Noetherian rings have finitely many associated prime ideals. But, there are examples given in \cite{Si}, \cite{Ka}, and \cite{SS}, which show that the set of associated primes of a local cohomology module $H^i_I (R)$ of a Noetherian ring $R$ with support in an ideal $I$ can be infinite. But, it remains an open question whether the set of primes minimal in the support of such local cohomology module is always finite \cite{HKM}, which is equivalent to the question whether support of such local cohomology modules are always Zariski closed. 

Let $R\rightarrow S$ be an arbitrary ring extension of Noetherian rings. In this article we study the behaviour of Zariski closedness of the support of local cohomology modules, more generally of Lyubeznik functors, when the ring extension $R\rightarrow S$ is namely `flat', `faithfully flat', `pure' and lastly `cyclically pure'. Here, we account for a brief description of Lyubeznik functor: Let $\Z$ be a closed subset of $\Spec R$ and $M$ be an $R$-module. We set $H^i_{\Z} (M)$ as the $i$-th local cohomology module of $M$ with support in $\Z$. We notice that $H^i_{\Z} (M)=H^i_I (M)$, for $\Z=\V(I)=\{P\in\Spec R: I\subset P\}$. For any two closed subsets of $\Spec R$, $\Z_1\subset \Z_2$, there is a long exact sequence of functors

$$\ldots\to H^i_{\Z_1}\to H^i_{\Z_2}\to H^i_{\Z_1/\Z_2}\to \ldots$$

We set $\T =\T_1\circ \dots\circ\T_t$, where every functor $\T_j$ is either $H^i_{\Z}$ for some closed subset $\Z$ of $\Spec R$ or the kernel or image (or cokernel) of some map in the above long exact sequence. $\T$ is known as Lyubeznik functor.

In section 2, at first, we extend the `Bourbaki formula' for the associated primes of flat extensions to that of the minimal associated primes and then we use it to study flat and faithfully flat ring extensions. As for example, for a faithfully flat extension, we obtain that $\Supp_{S} \T(S)$ is a Zariski closed subset of $\Spec S$ if and only if $\Supp_{R} \T(R)$ is a Zariski Closed subset of $\Spec R$. We also present few applications of the resuts of above `flat and faithfully flat' situations.

In section 3, we study the behaviour of the support under pure and cyclically pure ring extension. In Theorem 3.4, for an arbitrary module over the base ring, we observe how the Zariski closedness of the support of local cohomology modules can be transferred from the extension ring to its pure subring. In Theorem 3.6, we show that if base ring is local, then Zariski closedness of the support of Lyubeznik functor $\T$ can be transferred from the ring to its pure subring. Then, in Theorem 3.8, we show that under mild conditions on the rings, Zariski closedness of the support of Lyubeznik functor $\T$ can be transferred from the ring to its cyclically pure local subring. 

In section 4, we focus on the behaviour of the support for a special case of pure extension i.e. when the base ring is a direct summand of the extension ring. For flat and faithfully flat ring extension $R\rightarrow S$, for Lyubeznik functor $\T$, $\T(R)\otimes_R S= \T(S)$, but in general, we do not have any significant relation between them. In section 5, we continue to assume the ring extensions of section 4 and for such ring extensions, we compare the sets $\Supp_S(\T(R)\otimes_R S)$ and $\Supp_S \T(S)$. Important results are stated in Theorem 2.4, Corollary 1, Corollary 2, Corrolary 3, Theorem 3.4, Theorem 3.6, Theorem 3.8, Proposition 4.1, Theorem 4.2, Proposition 5.1 and finally in Proposition 5.2.

Throughout this paper, all the rings are assumed to be commutative Noetherian with unity. For basic results and for unexplained terms, we refer \cite{BH}, \cite{CAofM} and \cite{CRTofM}.
\section{behaviour of zariski closedness of the support under flat and faithfully flat ring extensions}

In this section, we observe the behaviour of Zariski closedness of the support of Lyubeznik functors and local cohomology modules under flat and faithfully flat ring extensions. For flat extension of Noetherian rings, Zariski closedness property of the support can be transferred from base ring to the extended ring, see (1) of Theorem 2.4. But, for faithfully flat extension we can also do the reverse, see Corollary 1. Actually, in (2) of Theorem 2.4, we get a slightly sharper result for the reverse and from which Corollary 1 follows. We also present a few applications of these results. 

\subsection{Main Results}

Consider flat ring homomorphism $R\rightarrow S$ of Noetherian rings where we reserve the symbols $\p$, $\p'$ for the elements of $\Spec R$ while $P$, $P'$ for the elements in $\Spec S$. Let $E$ be an $R$-module. For $\Ass_R E$, $\Ass_S (E\otimes_R S)$, and $\Ass_S (S/\p S)$, we denote the sets of minimal associated primes by $\min_R E$, $\min_S (E\otimes_R S)$, and $\min_S (S/\p S)$, 

\begin{lemma}
Let $R\rightarrow S$ be a flat ring homomorphism of Noetherian rings and let $\p_1$, $\p_2$ be the elements of $\Spec R$ such that $\p_1 S$ and $\p_2 S$ are proper ideals of $S$. Then the following are equivalent:\newline
(1) $\p_1=\p_2$\newline
(2) $\Ass_S(S/\p_1 S) = \Ass_S(S/\p_2 S)$\newline
(3) $\min_S(S/\p_1 S) = \min_S(S/\p_2 S)$ 
\end{lemma}

\begin{proof}
$(1)\Rightarrow (2)$ is easy.

For $(2)\Rightarrow (3)$, consider a minimal prime $P$ over $\p_1 S$ which is a minimal element of $\Ass_S(S/\p_1 S)$. We claim that it is also a minimal over $\p_2 S$ i.e. a minimal element of $\Ass_S(S/\p_2 S)$. If not, then there exists $P'$ such that $\p_2 S\subset P'\subset P$. Since $\Ass_S(S/\p_1 S) = \Ass_S(S/\p_2 S)$ and $P'\in\Ass_S(S/\p_1 S)$, this contradicts the fact that $P$ is minimal in $\Ass_S(S/\p_1 S)$. This implies that $\min_S(S/\p_1 S)\subset\min_S(S/\p_2 S)$ and by symmetry of the argument, we find that $\min_S(S/\p_1 S) = \min_S(S/\p_2 S)$. 

For $(3)\Rightarrow (1)$, take contraction of minimal prime. Since for any $P\in\Ass_S (S/\p S)$, $P\cap R= \p$ (see (i) of Theorem 12 of \cite{CAofM}), the result follows.
\end{proof}

In the following proposition, we extend the `Bourbaki formula' (see (ii) of Theorem 12 of \cite{CAofM}) for the minimal associated primes of flat extensions.

\begin{prop}
Let $R\rightarrow S$ be a flat ring homomorphism of Noetherian rings $R, S$ and $E$ be a $R$-module. Then, $\min_S (E\otimes_R S)= \bigcup_{\p\in \min_R E} \min_S (S/\p S)$.
\end{prop}

\begin{proof}
Since $R\rightarrow S$ is a flat ring homomorphism of Noetherian rings, from (ii) of Theorem 12 of \cite{CAofM}, we find that $\Ass_{S}(E\otimes_{R} S)= \cup_{\p\in \Ass_{R}E} \Ass_S(S/\p S)$. Let $P\in \min_{S}(E\otimes_{R} S)$. This implies that $P\in \Ass_S (S/\p S)$. If there exists some $P'\subset P$ and $P'\in \Ass_S (S/\p S)$, then $P'\in \Ass_{S}(E\otimes_{R} S)$ and this contradicts the fact that $P\in \min_{S}(E\otimes_{R} S)$. Thus, $P\in \min_{S}(E\otimes_{R} S)$ implies $P\in \min_S (S/\p S)$ for some $\p\in \Ass_R E$. We claim that $\p\in \min_R E$. If not, then there exists $\p'\subset \p$ such that $\p'\in \Ass_R E$ and we already have $P\in \min_S(S/\p S)$ with $P\cap R= \p$. By going down theorem, there exists $P'\subset P$, with $P'\cap R= \p'$. Thus, we have $\p' S\subset P'$. We can take $P'$ as the minimal prime over $\p'S$ and thus $P'\in \Ass_S(S/\p' S)$ (in fact $P'\in \min_S(S/\p' S)$. But, this gives that $P, P'\in \Ass_{S}(E\otimes_{R} S)$ and thus $P$ is not a minimal prime. Thus, we obtain that $\min_S (E\otimes_R S)\subset\bigcup_{\p\in \min_R E} \min_S (S/\p S)$.

For the other inclusion, consider $\p\in\min_R E$ and let $P\in \min_S(S/\p S)$. This implies $P\in\Ass_S (E\otimes_R S)$. Assume, there is some $P'\subset P$ and $P'\in\Ass_S (E\otimes_R S)$. Then $P'\in\Ass_S (S/\p' S)$ for some $\p'\in \Ass_R E$ and we can choose $P'$ in the $\min_S (S/\p' S)$. Thus, we have $\p= P\cap R\supseteq P'\cap R= \p'$. The last equality is due to an equivalent statement of going down theorem, see (\cite{CAofM}, page 31-32). But, if $\p=\p'$, then from above lemma we get, $\Ass_S(S/\p S) = \Ass_S(S/\p' S)$ and $P'\in \min_S(S/\p S)$, which contradicts the fact that $P\in \min_S(S/\p S)$. Further $\p'\subset \p$ implies that $\p$ is not in $\min_R E$, which is again a contradiction. Thus $P\in\min_S (E\otimes_R S)$ and we conclude for the other inclusion. So $\min_S (E\otimes_R S)= \bigcup_{\p\in \min_R E} \min_S (S/\p S)$.
\end{proof}

We state the following well known result. For the sake of completeness we present the proof.

\begin{lemma}
Let $R$ be a Noetherian ring and $M$ be an arbitrary $R$-module. Then, $\min_R M$ is finite if and only if $\Supp_R M$ is a Zariski closed set of $\Spec R$.
\end{lemma}

\begin{proof}
`if' is straight forward. For `only if', let $\p_1,\ldots,\p_n$ are $\min_R M$ and $I=\cap^n_{i=1} \p_i$. We first note that $\Supp_R M\subset \V(I)$. We prove the other inclusion to conclude. Clearly, for any $\p$, $\p\supseteq I$ implies $\p\supseteq \p_i$ for some $\p_i$. Assume that for some $\p\supset I$, such that $\p\supseteq \p_i$, let $M_p=0$. But, $0\neq M_{\p_i}= M\otimes_R R_{\p_i}= M\otimes_R (R_p)_{\p_iR_{\p_i}}= M\otimes_R R_{\p}\otimes_{R_{\p_i}}(R_p)_{\p_iR_{\p_i}}= M_{\p}\otimes_{R_{\p_i}}(R_p)_{\p_iR_{\p_i}}=0$, since $M_{\p}=0$. This is a contradiction. Thus $\Supp_R M= \V(I)$ and we conclude.  
\end{proof}

\begin{rem}
From Lemma 2.3 it is clear support of a local cohomology module is Zariski closed if we have the finiteness property of its set of associated primes.
\end{rem}

\begin{thm}
Suppose $R\rightarrow S$ be a flat ring homomorphism of Noetherian rings. \newline
(1) For $R$-module $E$, if $\Supp_R E$ is a Zariski closed subset in $\Spec R$ then $\Supp_S (E\otimes_R S)$ is a Zariski closed subset in $\Spec S$.\newline 
Consider an $R$-module $M$. For ideal $I\subset R$ and for $i\geq 0$, $\Supp_{R} H^i_I(M)$ is a Zariski closed subset of $\Spec R$ implies that $\Supp_{S} H^i_{IS}(M\otimes_R S)$ is a Zariski closed subset of $\Spec S$.\newline 
In particular, for Lyubeznik functor $\T$, if $\Supp_{R} \T(R)$ is a Zariski closed subset in $\Spec R$ then $\Supp_{S} \T(S)$ is a Zariski closed subset in $\Spec S$.\newline
(2) Converse will be true if in addition to we assume that only for finitely many primes $\p\in \min_{R}E$, $\p S$ is not a proper ideal of $S$. For local cohomology module and for Lyubeznik functor, converse will be true if in addition to we assume that only for finitely many primes $\p\in \min_{R} H^i_I(M)$ and $\p\in \min_{R} \T(R)$, $\p S$ is not a proper ideal of $S$.

\end{thm}

\begin{proof}
From Lemma 2.3, we can only focus on the finiteness property of the set of minimal associated primes. 

(1) For the first assertion, from Proposition 2.2 above, we get $\min_{S}(E\otimes_{R} S)= \cup_{\p\in \min_{R}E} \min_S(S/\p S)$. This gives that $\min_{S}(E\otimes_{R} S)$ is finite set if $\min_{R} E$ is a finite set, since $\min_S(S/\p S)$ is always finite. For second assertion, we set $E= H^i_I(M)$ in the above formula. For third assertion, set $E= \T(R)$ and from Lemma 3.1 of \cite{Ly1} we know that $\T(S)= \T(R)\otimes_R S$. Thus the result follows.

(2) For the converse, at first we assume that for every $\p\in \min_{R} E$, $\p S$ is a proper ideal of $S$. We observe the following: (1) For every $\p$ in $\Spec R$, $\p S$ is a proper ideal of $S$ if and only if $S/\p S$ is a non-zero $S$-module. Equivalently, this means that for any $\p\in \Spec R$, $\Ass_S(S/\p S)\neq \phi$ and equivalently $\min_S(S/\p S)\neq \phi$. (2) From above lemma we get that, $\min_S(S/\p_1 S) = \min_S(S/\p_2 S)$ if and only if $\p_1=\p_2$. (3) Consider a collection of finite sets $\{G_i\}_{i\in\Gamma}$ such that $G_i\neq\phi$ for every $i\in\Gamma$ and for every pair of $i,j\in \Gamma$, $i=j$ if and only if $G_i =G_j$. Then, finiteness of $G=\cup_{i\in\Gamma}G_i$ implies that $\Gamma$ is finite. This is because for finite $G$ with $|G|=n$ (say), there can be at the most $2^n$ number of distinct subsets of $G$. This gives $\Gamma$ is finite.

Now consider the following equation from the first assertion: $\min_{S}(E\otimes_{R} S)= \cup_{\p\in \min_{R}E} \min_S(S/\p S)$. From hypothesis we get that for every $\p$ in $\min_{R} E$, $\min_S(S/\p S)\neq\phi$. Thus, from above observations we find that if $\min_S (E\otimes_{R} S)$ is finite, then $\min_{R} E$ is finite. For second assertion, we set $E= H^i_I(M)$ and for third assertion set $E= \T(R)$

In general, for any $R$-module $E$, we can write $\min_{R} E= S_1 \cup S_2$, where the prime ideals of $S_1$ expand in $S$ as proper ideals and those of $S_2$ expand in $S$ as the whole ring. Since $S_2$ is finite, the result follows from above paragraph. 
\end{proof}

\begin{cor}
Let $S$ be faithfully flat over $R$.\newline
(1) For $R$-module $E$, $\Supp_S (E\otimes_{R} S)$ is Zariski closed subset in $\Spec S$ if and only if $\Supp_R E$ is Zariski closed subset in $\Spec R$.\newline
(2) In particular, for ideal $I\subset R$ and for $i\geq 0$, $\Supp_{S} H^i_{IS}(S)$ is a Zariski closed subset of $\Spec S$ if and only if $\Supp_{R} H^i_I(R)$ is a Zariski Closed subset of $\Spec R$.\newline
(3) More generally, for Lyubeznik funtor $\T$, $\Supp_{S} \T(S)$ is a Zariski closed subset of $\Spec S$ if and only if $\Supp_{R} \T(R)$ is a Zariski closed subset of $\Spec R$.\newline
\end{cor}

\begin{proof}
If $S$ is faithfully flat over $R$, then for every $\p$ in $\Spec R$, $\p S$ is a proper ideal of $S$. Thus, (1) and (2) follows immediately from above theorem.
For (3), we observe that $\T(S)= \T(R)\otimes_R S$ (see Lemma 3.1 of \cite{Ly1}), and result follows from (1).
\end{proof}

\begin{example}
There are many situations where we can apply above Corollary. Here we mention few of them.

(1) For any ring $R$, consider the polynomial ring $S=R[x_1,\dots, x_n]$ over $R$. From \cite{HKM}, we find for every $i\geq 0$ and for every $I\subset R$, $\Supp_{R} H^i_I(R)$ is a Zariski Closed subset of $\Spec R$ when $R$ is local and of dimension at the most four. Thus above Corollary implies that $\Supp_{S} H^i_{J}(S)$ is a Zariski Closed subset of $\Spec S$ for every expanded ideal os $S$ and where dimension of $S$ is $n+4$. 

(2) In invariant theory there are many situations where for certain subgroup $G$ of group of automorphisms of $R$, $R^G\rightarrow R$ is a faithfully flat ring extension. As for example, consider a finite subgroup $G$ of group of automorphism of $R$ such that group action is Galois (see \cite{AG} for the definition of Galois group action). Since $R$ is finite projective $R^G$-module, $R$ is faithfully flat over $R^G$. Now, if for every $i\geq 0$ and for every $J\subset R$, $\Supp_{R} H^i_J(R)$ is a Zariski Closed subset of $\Spec R$ (as for example, this can happen when $R$ is regular), then from above Corollary we get that, for every $i\geq 0$ and for every $I\subset R^G$, $\Supp_{R^G} H^i_I(R^G)$ is a Zariski Closed subset of $\Spec R^G$. 

In particular, for regular ring, the above situation is realizable in (a) of Proposition 4.2 in \cite{LP}.  
\end{example}

For local cohomology modules we have the following lemma. The results are well known result, but for the sake of completeness we present the proof.

\begin{lemma}
Consider an $R$-module $M$ and let $I$ be an ideal of Noetherian ring $R$. For $i\geq 0$, let $H^i_{I}(M)$ be its $i$-th local cohomology module. Then, $\Ass_{R}H^i_{I}(M)\subset \Supp_{R}H^i_{I}(M)\subset V(I)$. 
\end{lemma}

\begin{proof}
Since $\Ass_{R}H^i_{I}(M)\subset \Supp_{R}H^i_{I}(M)$, it is sufficient to show $\Supp_{R}H^i_{I}(M)\subset V(I)$. For $\p\in \Spec R$, assume that $(H^i_I (M))_p\neq 0$, then $\p$ should contain $I$, otherwise, we get some element $x\in I$ outside $\p$ which is unit in $R_{\p}$. Since, every element of $H^i_I (M)$ is annihilated by some power of $I$, every element of $(H^i_I (M))_p$ is annihilated by some power of $x$. This implies that $(H^i_I (M))_p = 0$. Thus $\Ass_{R}H^i_{I}(M)\subset \Supp_{R}H^i_{I}(M)\subset V(I)$. 
\end{proof}

When we apply the above results of flatness to localization, we can apply them to local cohomologies of any module over the ring. More precisely, for a Noetherian ring $R$, let $M$ be a finitely generated $R$-module. Consider the multiplicatively closed set $W$ of $R$. Due to flatness, Zariski closedness of the support of local cohomology modules is transferred from $R$ to $R_W$. But, using Lemma 2.5 above, we can have partial results that when we can transfer the Zariski closedness of the support from $R_W$ to $R$, see (2) and (3) of proposition below. 

\begin{prop}
Let $R$ be a Noetherian ring and $M$ be an $R$-module. Consider the multiplicatively closed set $W$ of $R$. Then we have the following:\newline
(1) For every $i\geq 0$ and for every ideal $I\subset R$, $\Supp_{R} H^i_I(M)$ is a Zariski closed subset of $\Spec R$ implies that $\Supp_{R_W} H^i_{IR_W}(M_W)$ is a Zariski closed subset of $\Spec R_W$.\newline 
In particular, if for every $i\geq 0$ and for every ideal $I\subset R$, $\Supp_{R} H^i_I(M)$ is a Zariski closed subset of $\Spec R$ then for every $i\geq 0$ and for every ideal $J\subset R_W$, $\Supp_{R_W} H^i_J(M_W)$ is a Zariski closed subset of $\Spec R_W$.\newline
(2) For $i\geq 0$, let only finite number of primes of $\min_RH^i_I(M)$ intersect with $W$. If $\Supp_{R_W} H^i_{IR_W}(M_W)$ is a Zariski closed subset of $\Spec R_W$ then so is $\Supp_RH^i_I(M)$ in $\Spec R$.\newline
In particular, for every $i\geq 0$, if only finite number of primes of $\V(I)$ intersects with $W$, then Zariski closedness of $\Supp_{R_W}H^i_{IR_W}(M_W)$ implies the Zariski closedness of $\Supp_R H^i_I(M)$.\newline
(3) Let $W=\{a^n: a\in R, n\in {\ZZ}_+\}$. For $i\geq 0$, let $\Ass_RH^i_I(M)$ be contained in the open subset $D(a)=\Spec R-V(a)$. If $\Supp_{R_W} H^i_{IR_W}(M_W)$ is a Zariski closed subset of $\Spec R_W$ then so is $\Supp_RH^i_I(M)$ in $\Spec R$.\newline
In particular, for every $i\geq 0$, if $\V(I +(a)R)$ is finite, then Zariski closedness of $\Supp_{R_W}H^i_{IR_W}(M_W)$ implies the Zariski closedness of $\Supp_R H^i_I(M)$.
\end{prop}

\begin{proof}
(1) First assertion follows from above Proposition. Second assertion is immediate, since every ideal in $R_W$ is an extended ideal. 

(2) Suppose, $I\cap W\neq\phi$. Lemma 2.5 implies that, every elements of $\Ass_{R}H^i_{I}(M)$ contains $I$ and here it also intersects with $W$. This implies, $\Ass_{R}H^i_{I}(M)$ is finite, since from hypothesis we have only finite number of primes which can intersect $W$. This implies that in this case $\Supp_R H^i_{I}(M)$ is Zariski closed, see Lemma 2.3 above. Thus, we can assume $I\cap W=\phi$. Moreover, using Lemma 2.3 again, we can focus only on the set of minimal associated primes. Set $J= IR_W$. From previous result in (1) we get, $\min_{R_W}(H^i_J(M_W))= \min_{R_W}(F\otimes_{R} R_W)= \cup_{\q\in \min_{R}F} \min_{R_W}(R_W/\q R_W)$, where $F= H^i_I(M)$. Here both $\min_{R_W}(H^i_J(M_W))$ and $\min_{R_W}(R_W/\q R_W)$ are finite sets. Observe that for all but finitely many $\q\in \min_{R}F$, $\Ass_{R_W}(R_W/\q R_W)=\{\q R_W\}\neq \phi$, since for all but finitely many $\q\in \min_{R}F$ does not intersect with $W$ and $\q R_W$ is a prime in $R_W$. Moreover, for $\q_1, \q_2 \in \min_{R}F$, $\q_1= \q_2$ if and only if $\min_{R_W}(R_W/\q_1 R_W)= \min_{R_W}(R_W/\q_2 R_W)$. Thus we can conclude that $\min_{R}F$ is finite. 

From Lemma 2.5 the second assertion is immediate.

(3) Both the assertions are immediate from (2) above.
\end{proof}

When $M$ is $R$, for more comprehensive situation than that of above Proposition 2.6, see Corollary 3.

\begin{rem}
From Lemma 2.5 above, we already know that $\Ass_{R}H^i_{I}(R)\subset\V(I)$. For $W=\{a^n: a\in R, n\in {\ZZ}_+\}$, one of the way the above Proposition can be realizable is to take $\V(I)\cap \V((a))$ is finite. Since, $\V(I)\cap \V((a))=\V(I+(a))$ we can give two examples of our desired situations:

(1) $\rad (I+(a))$ is the Jacobson radical of a semi-local ring,

(2) $(I+(a))$ is the whole ring.
\end{rem}

\subsection{Applications}

We recall the definition of regular algebra, see (\cite{CAofM}, page 249). Here we call them as smooth algebra. We observe the following for the smooth algebra.

\begin{lemma}
If $B$ is a smooth $A$-algebra, then for any multiplicatively closed set $W$ of $A$, $B_W$ is a smooth $A_W$-algebra.
\end{lemma}

\begin{proof}
Clearly $B_W$ is flat over $A_W$. Consider $P\in \Spec A_W$ with $P=\p A_W$ for some $\p\in \Spec A$. Now $\kappa(\p A_W)=(A_W)_{\p A_W}/(\p A_W)(A_W)_{\p A_W}= A_{\p}/\p A_{\p}=\kappa(\p)$. Thus for any finite extension $L$ of $\kappa(\p A_W)$, $B_W\otimes_{A_W} L= (B_W\otimes_{A_W} \kappa(\p A_W))\otimes_{\kappa(\p A_W)} L= ((B\otimes_A A_W)\otimes_{A_W} \kappa(\p A_W))\otimes_{\kappa(\p A_W)} L= B\otimes_A \kappa(\p)\otimes_{\kappa(\p)} L= B\otimes_A L$. Since $B\otimes_A L$ is regular, we conclude.
\end{proof}

Here, we adopt the following terminology: We say a Noetherian ring $R$ satisfies `Zariski closedness of the support' if for every ideal $I\subset R$ and for every $i\geq 0$, $\Supp_R H^i_I (R)$ is a Zariski closed subset of $\Spec R$. For primes $a_1,\ldots,a_n$ in $\ZZ$, let $\ZZ_{{a_1},\dots,{a_n}}$ be the ring where every prime of $\ZZ$ is inverted except the prime ideals $a_1,\ldots,a_n$. As for example, consider the primes $a$ and $b$ of $\ZZ$. Clearly, for any two elements $x,y \in \ZZ$, $x,y \notin a\ZZ \cup b\ZZ$, implies that $xy \notin a\ZZ \cup b\ZZ$ and thus we have $\ZZ_{{a},{b}}$. In general, consider a Noetherian semilocal domain $V_{a_1,\ldots,a_n}$ of characteristic zero with $n$ maximal ideals, where each of them is generated by each of the prime integers $a_1,\ldots,a_n$. Clearly $\ZZ_{{a_1},\dots,{a_n}}$ is an example of such a semi local domain and it also sits inside $V_{a_1,\ldots,a_n}$. $V_{a_1,\ldots,a_n}$ is a semilocal ring of mixed characteristic which is a generalization of $p$-ring (\cite{CRTofM}, page 223), which is a Noetherian local domain of characteristic zero where the maximal ideal is generated by some prime integer of $\ZZ$. 

Theorem 2.4 above, can be applicable in the following corollary where we show that any smooth algebra $R$ over $V_{a_1,\ldots,a_n}$ satisfies `Zariski closedness of the support'. 

\begin{cor}
Here we adopt the notation of above paragraphs.\newline
(1) Let $a_1,\ldots,a_n$ be prime integers in $\ZZ$, Let $R$ be a finitely generated flat $V_{a_1,\ldots,a_n}$-algebra containing $\ZZ_{a_1,\ldots,a_n}$ as a subring of it. For any prime $a\in \ZZ$, if every finitely generated flat algebra over $V_a$ satisfies `Zariski closedness of the support', then for every $i\geq 0$ and for every $I\subset R$, we have $\Supp_R H^i_I (R)$ is a Zariski closed subset of $\Spec R$.\newline 
(2) In particular, let $a_1,\ldots,a_n$ be prime integers in $\ZZ$, Let $R$ be a smooth $V_{a_1,\ldots,a_n}$-algebra containing $V_{a_1,\ldots,a_n}$ as a subring of it. Then for every $i\geq 0$ and for every $I\subset R$, we have $\Supp_R H^i_I (R)$ is a Zariski closed subset of $\Spec R$. 
\end{cor}

\begin{proof}
(1) We proceed by induction on $n$. When $n=1$, the result follows from hypothesis. So, we consider the case $n=2$. Set $a_1 =a$ and $a_2= b$. According to the hypothesis, in $R$, $a$ and $b$ are coprimes. For any $I\subset R$, if it contains $b$, we always have $I+ (a)R= R$. Thus $\V(I)\cap \V((a))= \V(I+(a))=\phi$ and $I$ cannot contain both $a$ and $b$. So, we assume that $b\in I$. On the other hand, from Lemma 2.5 we get, $\Ass_{R}H^i_{I} (R)\subset \Supp_{R} H^i_{I}(R)\subset V(I)$. This implies that no prime of $\min_{R} H^i_I (R)$ intersects with $W_a=\{a^n: n\in {\ZZ}_+\}$. 

Now, $V_b$ is localization of $V_{a, b}$ for the multiplicatively closed set $W_a$. Since $R$ is a finitely generated flat $V_{a, b}$-algebra, by above Lemma $R_{W_a}$ is also a finitely generated flat $V_b$-algebra. From hypothesis we find that $\Supp_{R[1/a]} H^i_J(R[1/a])$ is a Zariski closed subset of $\Spec R[1/a]$ for every $i\geq 0$ and for every ideal $J\subset R[1/a]$. Thus from (2) of above Proposition 2.6, we have $\Supp_R H^i_I (R)$ is a Zariski closed subset of $\Spec R$, for every $i\geq 0$, where $I\subset R$ contains $b$. Now from the symmetry in the above argument, the same holds when $a\in I$.  

Finally, assume neither $a$ nor $b$ is in $I$. Consider the multiplicatively closed set $W_a=\{a^n: n\in {\ZZ}_+\}$ and $W_b=\{b^n: n\in {\ZZ}_+\}$. If there are finitely many primes of $\min_R H^i_I(R)$ intersects $W_a$ or $W_b$, then using (2) of Proposition 2.6, we find $\Supp_R H^i_I (R)$ is a Zariski closed subset of $\Spec R$. Otherwise, each of the elements $a$ and $b$ are in infinitely many primes from $\min_R H^i_I (R)$. Now, no prime can contain both $a$ and $b$ since they are coprimes. So there are infinitely many primes of $\min_R H^i_I (R)$ which do not contain $a$ as well and similar is true for $b$. We observe the following: for $a\in R$, $\p\in \Ass_R H^i_I (R)$ and $a\notin \p$ if and only if $\p R[1/a]\in \Ass_{R[1/a]} H^i_{IR[1/a]} (R_{R[1/a]})$. Since this 1:1 correspondence is order preserving, it implies that, for any $a\in R$, $\p\in \min_R H^i_I (R)$ and $a\notin \p$ if and only if $\p R[1/a]\in \min_{R[1/a]} H^i_{IR[1/a]} (R[1/a])$. From hypothesis $\min_{R[1/a]} H^i_{IR[1/a]} (R[1/a])$ is finite. This implies that the subset of primes of $\min_R H^i_I (R)$ which does not contain $a$, is finite. This yields a contradiction and this completes the proof for the case $n=2$. 

Assume the result is true for $n-1$ i.e. any finitely generated flat algebra $R$ over $V_{{a_1},\ldots,{a_{n-1}}}$ satisfies `Zariski closedness of the support'. Consider a finitely generated flat algebra over $V_{{a_1},\ldots,{a_{n}}}$ and proceeding in the similar way as in the $n=2$ case, we can show that it also satisfies `Zariski closedness of the support'. This concludes the proof.  

(2) From Theorem 4.1 of \cite{BBLSZ} we know that every smooth algebra over $V_a$ satisfies `Finiteness condition of associated primes' and this implies `Zariski closedness of the support'. Now, we can use Lemma 2.7 and the result follows from (1). 
\end{proof}

The (2) of Corollary 2 above, can be proved from (2) of Theorem 3.2 of \cite{Bh}, but for completeness we present its proof.

In Proposition 2.6 we observe that, for flat extension of Noetherian rings, Zariski closedness of the support can be transferred from base ring to the extended ring. We can do the reverse for faithfully flat situation. Localization is never faithfully flat unless we localize the ring trivially, i.e. by the set of units. Due to the flatness, in (1) of Proposition 2.6, we observe that `Zariski closedness of the support' goes up from base ring to its localization. In (2) and (3) of Proposition 2.6 we have partial results that how Zariski closedness can come down from localized ring to base ring. The following corollary gives more comprehensive result.

\begin{cor}
Let $R$ be a Noetherian domain and for $a\in R$ consider the localized ring $R[1/a]$. If for every non unit $a\in R$, $R[1/a]$ satisfies `Zariski closedness of the support', then for every $i\geq 0$ and for every ideal $I\subset R$, $\Supp_R H^i_I (R)$ is a Zariski closed subset in $\Spec R$.\newline 
In particular, let $R$ be a Noetherian local domain which is finitely generated over a field. Let for every prime $\p$ of the punctured spectrum of $R$, $R_{\p}$ is regular, then for an ideal $I\subset R$ and for $i\geq 0$, $\Supp_R H^i_I (R)$ is a Zariski closed subset in $\Spec R$.
\end{cor}

\begin{proof}
From the proof of Proposition 2.6, we already know that $\min_{R}H^i_{I}(R)\subset \Ass_{R}H^i_{I}(R)\subset \Supp_{R}H^i_{I}(R)\subset \V(I)$. If we can choose a nonunit $a$ such that $a\notin \p$ and $\p$ is in some infinite subset of $\min_R H^i_I (R)$, then using (2) of Proposition 2.6, we find that $\Supp_R H^i_I (R)$ is a Zariski closed subset in $\Spec R$. Otherwise, for every element $a$ of every maximal ideal, there exists infinitely many primes from $\min_R H^i_I (R)$ such that $a$ is in every prime. Fix one maximal ideal $\m$ and suppose $a_1,\ldots, a_n$ generates $\m$. We first observe the following: for any $a\in R$, $\p\in \Ass_R H^i_I (R)$ and $a\notin \p$ if and only if $\p R[1/a]\in \Ass_{R[1/a]} H^i_{IR_{R[1/a]}} (R[1/a])$. Since this 1:1 correspondence is order preserving, it implies that, for any $a\in R$, $\p\in \min_R H^i_I (R)$ and $a\notin \p$ if and only if $\p R[1/a]\in \min_{R[1/a]} H^i_{IR[1/a]} (R[1/a])$. From hypothesis, since $\Supp_{R[1/a]} H^i_{IR[1/a]} (R[1/a])$ is Zariski closed subset in $\Spec R[1/a]$, $\min_{R[1/a]} H^i_{IR[1/a]} (R[1/a])$ is finite. This implies that the subset of primes of $\min_R H^i_I (R)$ which does not contain $a$, is finite.

Thus, for each $a_i$ there exists a finite subset ${F_i}$ of $\min_R H^i_I (R)$ such that $a_i\notin \p$ if and only if $\p\in {F_i}$. For any prime $\p\in \min_R H^i_I (R)$, $\m=(a_1,\ldots, a_n)R$ is not inside $\p$ if and only if  $a_1\notin \p$ or $a_2\notin \p$ or $\ldots$ or $a_n\notin \p$. Clearly there can be at the most finite such $\p\in\min_R H^i_I (R) $ where $\p$ does not contain $\m$. Since we assume $\min_R H^i_I (R)$ is infinite, there exists one $\p\in\min_R H^i_I (R)$ such that $\m\subset \p$. Thus $\m=\p$ and $\m\in \min_R H^i_I (R)$. Since this is true for every $\m$, we find that $\min_{R}H^i_{I}(R)=\V(I)=\Supp_{R}H^i_{I}(R)$. This gives that $\Supp_{R}H^i_{I}(R)$ is a Zariski closed subset in $\Spec R$.

For any non-unit $a\in R$, for every $\p$ such that $a\notin\p$, we have $R_{\p}$ is regular. So $R[1/a]$ is regular.  Thus using the results of \cite{HS} and \cite{Ly1}, the assertion follows from the first assertion.
\end{proof}

The last assertion of Corollary 3 above, is proved in Corollary 3 of \cite{Bh}, but for completeness we present its proof.   

\section{behaviour of zariski closedness of the support under pure and cyclically pure ring extensions}

In this section, we observe the behaviour of the Zariski closedness of the support of Lyubeznik functors as well as the local cohomology modules, under pure and cyclically pure ring extension. At first, we review basic results for pure and cyclically pure extensions. In Theorem 3.4, we observe that, how the Zariski closedness of the support of local cohomology module of an arbitrary module over the base rings, can be transferred from the ring to its pure subring. In Theorem 3.6, we show that if base ring is local, then Zariski closedness of the support of Lyubeznik functor $\T$ can be transferred from the ring to its pure subring. Then, in Theorem 3.8, we show that under mild conditions on the rings, Zariski closedness of the support of Lyubeznik functor $\T$ can be transferred from the ring to its cyclically pure local subring. 

We recall the definitions of pure and cyclically pure ring extension. Let $A\rightarrow B$ be an injective ring homomorphisms. For every $A$-module $M$, if $M\rightarrow M\otimes B$ is injective then we say $A$ is a pure subring of $B$ or the ring extension $A\rightarrow B$ is a pure ring extension. Moreover, if the map remains injective only after tensoring with every $A$-module of the form $A/I$ for some ideal $I\subset A$ or equivalently, if for every ideal $I\subset A$, $A/I\rightarrow B/IB$ is injective then we call the ring extension $A\rightarrow B$ as cyclically pure ring extension.

\begin{rem}
Clearly, purity implies cyclic purity, but the converse is true when the ring is Approximately Gorenstein, see \cite{Ho1}. Here we highlight few cases where the converse is true:

(1) for complete semilocal reduced ring, more generally for locally excellent reduced ring, 

(2) for normal domain.
\end{rem}

Here, we two possible sources of pure ring extensions:

\begin{example}
(1) If a ring extension is faithfully flat, then it is a pure extension.

(2) For a ring homomorphism $A\rightarrow B$, if $A$ is a direct summand of $B$, equivalently if the map splits then it is pure.

\end{example}

In the next sections, we will explore the example pure extension when the base ring is a direct summand of its extension ring and we study the behaviour of the Zariski-closedness of the support extensively.

\begin{rem}
Let $A\rightarrow B$ be a homomorphism of Noetherian rings. Let $M$ be a finitely generated $B$-module. From (\cite{CRTofM}, Exercise 6.7), we find that if $Ass_B M$ is a finite set then so is $Ass_A M$. For arbitrary $B$-module $M$ we have a proof given in the Proposition 2.2 of \cite{Pu}. This will be used in Lemma 3.1.
\end{rem}  

We observe the following lemmas for minimal primes. Lemma 3.1 and Lemma 3.3 will be used in Theorem 3.4.

\begin{lemma}
(1) For a Noetherian ring $R$, let $M$ be an $R$-module. Consider a submodule $N$ of $M$. Then, $\min_R N\subset \min_R M$.\newline 
(2) Let $R\rightarrow S$ be a ring homomorphism of Noetherian rings, and let $M$ be an $S$-module. Then, $\min_R M\subset \{P\cap R:P\in \min_S M\}$.
\end{lemma}

\begin{proof}
(1) For the first assertion, let $\p\in \min_R N$. So there exists $x\in N\subset M$ such that $\p x= 0$. Suppose there is $\q\subset \p$, such that $\q\in \Ass_R M$. But, then $\q$ also annihilates $x\in N$ and contradicts that $\p\in \min_R N$. Thus we conclude.

(2) By Proposition 2.2 of \cite{Pu}, we have $\Ass_R M=\{P\cap R:P\in\Ass_S M\}$. Let $\p\in\min_R M$. Then there exists $P\in\Ass_S M$ such that $\p=P\cap R$. We can further assume that $P$ is a minimal prime in $\Ass_S M$ such that $\p=P\cap R$. Let there exist another $P'\subset P$ such that $P'\in\Ass_S M$. This implies that $\p= P\cap R\supseteq P'\cap R$ and $(P'\cap R)\in \Ass_R M$. Since $\p$ is minimal we find $\p= P'\cap R$ which is a contradiction. Thus $P\in\min_S M$ and we get $\min_R M\subset \{P\cap R:P\in \min_S M\}$.
\end{proof}

\begin{lemma}
Let $R\rightarrow S$ be a pure extension of Noetherian rings, and $M$ be an $S$-module. Let $N$ be an $R$-submodule of $M$. If $\min_S (N\otimes_R S)\subset \min_S (M\otimes_R S)$, then $\min_R (N\otimes_R S)\subset \{P\cap R: P\in\min\Ass_S (M\otimes_R S)\}$.
\end{lemma}

\begin{proof}
From above Lemma we get $\min_R (N\otimes_R S)\subset \{P\cap R:P\in \min_S (N\otimes_R S)\}$ and $\min_R (M\otimes_R S)\subset \{P\cap R:P\in \min_S (M\otimes_R S)\}$. Thus for any $\p\in \min_R (N\otimes_R S)$, $\p= P\cap R$ for some $P\in \min_S (N\otimes_R S)$. But from hypothesis we already have $\min_S (N\otimes_R S) \subset \min_S (M\otimes_R S)$. This give $P\in \min_S (M\otimes_R S)$ and $\p\in\{P\cap R: P\in\min_S (M\otimes_R S)$\}. Thus $\min_R (N\otimes_R S)\subset \{P\cap R: P\in\min_S (M\otimes_R S)\}$.
\end{proof}

\begin{lemma}
Let $R\rightarrow S$ be a pure extension of Noetherian rings, and $M$ be an $S$-module. 
Then for any $R$-submodule $N\subset M$ we have $\min_R N\subset\min_R (N\otimes_R S) $.
\end{lemma}

\begin{proof}
Let $\p\in\min_R N$. Since $N\rightarrow N\otimes_R S$ is injective, $\p\in\Ass_R (N\otimes_R S)$. Let $\p\notin \min_R (N\otimes_R S)$ and there exists some $\p'\in\Ass_R (N\otimes_R S)$ such that $\p'\subset \p$ such that $\p'=\Ann_R y$ for some $y\in N\otimes_R S$. If $y=\sum^n_{i=1}(x_i\otimes_R s_i)$, we can take it to $M\otimes_R S$ and we can write $y$, as $y=\sum^n_{i=1}(s_ix_i\otimes_R 1)\in Im(N\otimes_R S\rightarrow M\otimes_R S)\subset (M\otimes_R S)$ since $M$ can be viewed as $S$-module. Thus, there exists $z=\sum^n_{i=1}s_ix_i \in M$ such that, $z\otimes_R 1=y$. So by Proposition 6.5 of \cite{HR}, $y$ is in $N$.  But this suggests that we can think $\p'$ as an element of $\Ass_R N$, which contradicts that $\p\in\min_R N$. Thus $\min_R N\subset \min_R (N\otimes_R S)$. 
\end{proof}

Now, we state our first main result of this section, regarding the behaviour of the Zariski closedness of the support of local cohomologies under pure base change. 

\begin{thm}
Let $R\rightarrow S$ a pure extension of Noetherian rings and let $I$ be an ideal of $R$. Consider an $R$-module $M$. For some $i\geq 0$ and for local cohomology module $H^i_{I}(M)$, if $\Supp_{S} (H^i_I (M)\otimes_R S)$ is a Zariski closed subset in $\Spec S$, then $\Supp_{R} H^i_{I}(M)$ is a Zariski closed subset in $\Spec R$. 
\end{thm}

\begin{proof}
From hypothesis, using Lemma 2.3, we get $\min_S (H^i_I(M)\otimes_R S)$ is finite. From (2) of Lemma 3.1 we get $\min_R (H^i_I(M)\otimes_R S)$ is also finite. Here $S$ is also an $R$ module, we can think $H^i_{IS}(M\otimes_R S)$ as $R$-module $H^i_{I}(M\otimes_R S)$ due to base change. For pure extension $R\rightarrow S$, using corollary 6.8 of \cite{HR}, $H^i_{I}(M)$ can be thought as $R$-submodule of $H^i_{I}(M\otimes_R S)$. So, we can apply Lemma 3.3 and we find that $\min_R H^i_I(M)\subset\min_R (H^i_I(M)\otimes_R S)$ and latter is finite. Thus $\min_R H^i_{I}(M)$ is also finite and thus again by Lemma 2.3, $\Supp_{R} H^i_{I}(M)$ is also a Zariski closed subset in $\Spec R$. 
\end{proof}

\begin{rem}
If $R\rightarrow S$ is faithfully flat, then it is pure and in that case $H^i_I (R)\otimes_R S = H^i_{IS}(S)$. Then, `only if' of (2) of Corollary 1 is immediate from Theorem 3.4. The `only if' part of the assertion is immediate from (2) of above Proposition 4.1. 

\end{rem}

In the following proposition, we consider the special situation of pure ring extension, i.e. when the base ring is a direct summand of the extension ring. We observe how Zariski closedness of the support of Lyubeznik functors as well as local cohomologies can be transferred from ring to its base ring.

\begin{prop}
Let $R\rightarrow S$ be an extension of Noetherian rings such that $R$ is a direct summand of $S$.\newline
(1) For Lyubeznik functor $\T$, if $\Supp_S \T(S)$ is a Zariski closed subset of $\Spec S$ then $\Supp_R \T(R)$ is a Zariski closed subset of $\Spec R$.\newline 
(2) In particular, for every $i\geq 0$, if $\Supp_{S} H^i_J(S)$ is a Zariski closed subset of $\Spec S$ for every ideal $J\subset S$, then for every $i\geq 0$, $\Supp_{R} H^i_{I}(R)$ is also a Zariski closed subset of $\Spec R$ for every ideal $I\subset R$. 
\end{prop}

\begin{proof}
(1) From Lemma 2.3 it is sufficient to consider the set of minimal primes. Since $R$ is a direct summand of $S$, we find $T(R)$ is also direct summand of $T(S)$ as an $R$-module. From (2) of Lemma 3.1, $\min_R \T(S)\subset \{P\cap R:P\in \min_S \T(S)\}$. Thus $\min_R \T(S)$ is finite, since from hypothesis we get that $\min_S \T(S)$ is finite. Since $T(R)$ is an $R$-submodule $T(S)$, from (1) of Lemma 3.1 we find that $\min_R \T(R)$ is finite. Thus $\Supp_R \T(R)$ is a Zariski closed subset of $\Spec R$.

(2) This is immediate from (1).
\end{proof}

Now, we state our next important result of this section, regarding the behaviour of the Zariski closedness of the support of Lyubeznik functors under pure base change, when the base ring is local. 

\begin{thm}
Let $R\rightarrow S$ be a pure extension of Noetherian rings, where $(R,\m)$ is a local ring.\newline
(1) For Lyubeznik functor $\T$, if $\Supp_S \T(S)$ is a Zariski closed subset of $\Spec S$ then $\Supp_R \T(R)$ is a Zariski closed subset of $\Spec R$.\newline 
(2) In particular, for every $i\geq 0$, if $\Supp_{S} H^i_J(S)$ is a Zariski closed subset of $\Spec S$ for every ideal $J\subset S$, then for every $i\geq 0$, $\Supp_{R} H^i_{I}(R)$ is also a Zariski closed subset of $\Spec R$ for every ideal $I\subset R$. 
\end{thm}

\begin{proof}
(1) For pure extension $R\rightarrow S$, consider the extension $\hat{R}\rightarrow \hat{S}$, where $\hat{S}$ is completion of $S$ in $\m$-adic topology. From Corollary 6.13 of \cite{HR}, we find that $\hat{R}\rightarrow \hat{S}$ is also pure. From Exercise 9.5 in page 64 of \cite{Hu2}, we find that extension $\hat{R}\rightarrow \hat{S}$ actually splits. For Lyubeznik funtor $\T$, consider $\T(R)$. Since $\Supp_{S} \T(S)$ is a Zariski closed subset of $\Spec S$, from faithfully flat ring extension $S\rightarrow \hat{S}$, $\Supp_{\hat{S}} \T(\hat{S})$ is also a Zariski closed subset of $\Spec \hat{S}$, see (3) of Corollary 1. Since, extension $\hat{R}\rightarrow \hat{S}$ actually splits, using Proposition 3.5 above, we have the closedness of $\Supp_{\hat{R}} \T(\hat{R})$. Now, again by (3) of Corollary 1, from the result of faithfully flat base change we observe that $\Supp_{R} \T(R)$ is a Zariski closed subset of $\Spec R$. Thus we conclude.

(2) It is immediate from (1).
\end{proof}

We need the following lemma to prove our next main result. This lemma is already proved in \cite{Bh}, for completeness, we present the proof.

\begin{lemma}
Let $R\rightarrow S$ be a cyclically pure ring extension of Noetherian rings, where $R$ is a local ring with maximal ideal $\m$. Then $\hat{R}\rightarrow \hat{S}$ is also cyclically pure, where $\hat{S}$ is completion of $S$ in $\m$-adic topology.
\end{lemma}

\begin{proof}
We observe the following fact: Consider a ring homomorphism $A\rightarrow B$ of Noetherian rings where $(A,\m)$ is a Noetherian local ring with maximal ideal $\m$. Let for every $\m$-primary ideal $\q\subset A$, $\q B\cap A=\q$, then for every ideal $I$ of $A$, $IB\cap A=I$ i.e. ring homomorphism is cyclically pure. 

To see this we mention the following fact: For any ideal $I\subset A$, $I$ can be written as an arbitrary intersection of $\m$-primary ideals. This is due to Krull-Intersection Theorem. As for example one can think $I=\bigcap^{\infty}_{n=1}(I+{\m}^n)$. Let $I=\bigcap_{i\in \Omega} {\q}_i$ where $\Omega$ is an arbitrary index set. Thus $I\subset IB\cap A\subset (\bigcap_{i\in \Omega} {\q}_i)B\cap A\subset (\bigcap_{i\in \Omega} ({\q}_iB))\cap A=\bigcap_{i\in \Omega} ({\q}_iB\cap A)=\bigcap_{i\in \Omega} {\q}_i=I$. This implies $IB\cap A=I$. 

Since $R\rightarrow S$ is cyclically pure, for every $\m$-primary ideal ${\q}_i$ we have ${\q}_iS\cap R={\q}_i$. Let $\hat{R}$ be the completion of $R$ in $m$-adic topology. Consider the following commutative diagram.  
\[
  \xymatrix
{
  & \hat{R} 
    \ar@{->}[r]
 & S\otimes\hat{R}=\hat{S}  
 \\
  & R
	\ar@{->}[u]
\ar@{->}[r]
 & S
  \ar@{->}[u]      
 }
\]
Going to the ring $\hat{R}$ with a maximal ideal $\hat{\m}$ via faithfully flat ring extension $R\rightarrow \hat{R}$ we find that every $\hat{\m}$-primary ideal $\hat{\q}_i$ is of the form $\hat{\q}_i={\q}_i\hat{R}$. Consider the ring homomorphism $\hat{R}\rightarrow \hat{S}=S\otimes \hat{R}$. Here ${\q}_iS\cap R={\q}_i$ is equivalent to $R/{\q}_i \rightarrow S/{\q}_i S$ is injective. Tensoring with $\hat{R}$ the last map remains injective i.e. we get $\hat{R}/\hat{\q}_i \rightarrow \hat{S}/\hat{{\q}_i}\hat{S}$ is injective. Thus for every $\hat{\m}$-primary ideal we have $\hat{{\q}_i}\hat{S}\cap \hat{R}=\hat{{\q}_i}$. Thus from above paragraph of the proof we get $\hat{R}\rightarrow \hat{S}$ is also a cyclically pure ring extension.
\end{proof}

We recall that a Noetherian semilocal ring is analytically unramified if its completion (by its Jacobson radical) is reduced see \cite{CAofM}. Now, we state the last important result of this section, regarding the behaviour of the Zariski closedness of the support of local cohomologies under cyclically pure base change.

\begin{thm}
Consider a cyclically pure ring homomorphism $R\rightarrow S$ of Noetherian rings where $R$ is a analytically unramifed Noetherian local ring with maximal ideal $\m$.\newline 
(1) For Lyubeznik functor $\T$, if $\Supp_S \T(S)$ is a Zariski closed subset of $\Spec S$ then $\Supp_R \T(R)$ is a Zariski closed subset of $\Spec R$.\newline 
(2) In particular, for every $i\geq 0$, if $\Ass_{S} H^i_J(S)$ is a Zariski closed subset of $\Spec S$, then for every $i\geq 0$, $\Ass_{R} H^i_{I}(R)$ is also a Zariski closed subset of $\Spec R$.  
\end{thm}

\begin{proof}
(1) From Lemma 3.7, we find that $\hat{R}\rightarrow \hat{S}$ is also cyclically pure. Since $\hat{R}$ is also reduced, using \cite{Ho1} or from above Remark 3, $\hat{R}\rightarrow \hat{S}$ is actually a pure extension. Moreover, from Exercise 9.5 in page 64 of \cite{Hu2}, we find that extension $\hat{R}\rightarrow \hat{S}$ actually splits. For Lyubeznik funtor $\T$, consider $\T(R)$. Since $\Supp_{S} \T(S)$ is a Zariski closed subset of $\Spec S$, from faithfully flat ring extension $S\rightarrow \hat{S}$, $\Supp_{\hat{S}} \T(\hat{S})$ is also a Zariski closed subset of $\Spec \hat{S}$, see (3) of Corollary 1. Since, extension $\hat{R}\rightarrow \hat{S}$ actually splits, using Proposition 3.5 above, we have the closedness of $\Supp_{\hat{R}} \T(\hat{R})$. Now, again by (3) of Corollary 1, from the result of faithfully flat base change we observe that $\Supp_{R} \T(R)$ is a Zariski closed subset of $\Spec R$. Thus we conclude.

(2) This assertion is immediate from (1).
\end{proof}

\section{behaviour of zariski closedness of the support under pure extension: examples}

There are two possible sources form where we can get examples of pure extensions. One of them is faithfully flat extension and in another case, base ring is a direct summand of extension ring. In this section, we focus on the following examples of pure extensions:

We assume base ring $R$ is a direct summand of extension ring $S$ as an $R$-module i.e. $S= R\oplus T$, but $T$ is also an ideal of $S$. Equivalently, one can assume that $T$ is a multiplicatively closed subset of $S$. We refer this situation as `extension $R\rightarrow S$ satisfies (\textbf{*})'. Under this special extension, we observe that for Lyubeznik functor $\T$, $\Supp_{R} \T(R)$ is a Zariski closed subset in $\Spec R$ if and only if $\Supp_{S} (\T(R)\otimes_R S)$ is a Zariski closed subset in $\Spec S$, see Theorem 4.2 (and also compare this result with that of Theorem 3.4). 

Here, we list few examples of `extension $R\rightarrow S$ satisfies (\textbf{*})'.

\begin{example}
(1) Let $S=\oplus^{\infty}_{i=0} S_i$ be a graded algebra over $R$ where $S_0= R$. Set $T= \oplus^{\infty}_{i\geq 0} S_i$. Clearly $T$ is also an ideal of $S$ and we find that extension $R\rightarrow S$ satisfies (\textbf{*}).

(1) When $S$ is an trivial extension of $R$, see page 191 of \cite{CRTofM}, extension $R\rightarrow S$ satisfies (\textbf{*}). 

(2) If $S$ is direct product of rings $R$ and $T$ i.e. $S= R\times T$ such that $T$ is also an $R$ algebra, then extension $R\rightarrow S$ satisfies (\textbf{*}). 
\end{example}

\begin{rem}
In the above situation when `extension $R\rightarrow S$ satisfies (\textbf{*})', $T$, being an ideal of $S$, we can take quotient ring $S/T$ and it turns out to be $R$. Moreover, identity map $R\rightarrow R$ can be factored through $R\rightarrow R\oplus T= S\rightarrow S/T= R$. Consider $\q\in \Spec S$ and set $\p= \q\cap R$ in $\Spec R$. Thus for multiplicatively closed set $W= S-\q$ of $S$, we get $W^{-1}R= R_{\p}$, together with the ring homo $S_{\q}\rightarrow R_{\p}$. Again in the same way identity map $R_{\p}\rightarrow R_{\p}$ can be factored through $R_{\p}\rightarrow R_{\p}\oplus T_{\p}= S_{\p}\rightarrow S_{\q}\rightarrow (S/T)_{\q}= R_{\p}$.
\end{rem}

It is well known that for arbitrary ring extension $R\stackrel{f}\rightarrow S$ and for finitely generated $R$-module $M$, $\Supp_S (M\otimes_R S)= {f^*}^{-1}(\Supp_R M)$, where we have the natural map $\Spec S\stackrel{f^*}{\rightarrow}\Spec R$, (see second assertion of Proposition 19 of Section II.4.4 of \cite{Bo}) 

In the following proposition, we see that above is also true for arbitrary $R$-module $M$, but in that case we have to take ring extension $R\rightarrow S$ satisfying (\textbf{*}) above.

\begin{prop}
Consider the following ring extension $R\rightarrow S$ satisfying (\textbf{*}). Then, the following assertions are true.

(1) For any arbitrary $R$-module $M$, $\Supp_S (M\otimes_R S)= {f^*}^{-1}(\Supp_R M)$.

(2) If $\Supp_R M$ is closed in $\Spec R$, then $\Supp_S(M\otimes_R S)$ is also closed in $\Spec S$.

(3) In particular, If $\Supp_R \T(R)$ is a Zariski closed subset in $\Spec R$, then $\Supp_S(\T(R)\otimes_R S)$ is also a Zariski closed subset in $\Spec S$.
\end{prop}

\begin{proof}
(1) Let $R\stackrel{f}{\rightarrow}S$ is a ring homomorphism such that $S= R\oplus T$. Consider the natural map $\Spec S\stackrel{f^*}{\rightarrow}\Spec R$. From first assertion of Proposition 19 of Section II.4.4 of \cite{Bo}, we find that $\Supp_S(M\otimes_R S)\subset {f^*}^{-1}(\Supp_R M)$. For the reverse inclusion, let $\q\in {f^*}^{-1}(\Supp_R M)$. Set $\p=\q\cap R= f^*(\q)\in \Spec R$ and we have $M_{\p}\neq 0$. Now consider $(M\otimes_R S)_{\q}= (M\otimes_R S)\otimes_S S_{\q}= M\otimes_R S_{\q}= M\otimes_R R_{\p}\otimes_{R_{\p}} S_{\q}= M_{\p}\otimes_{R_{\p}} S_{\q}$. Now observe the following: $M_{\p}= M_{\p}\otimes_{R_{\p}} R_{\p}= M_{\p}\otimes_{R_{\p}} S_{\q}\otimes_{S_{\q}} R_{\p}$ (see Remark 8 above). Thus $M_{\p}\otimes_{R_{\p}} S_{\q}= 0$ implies that $M_{\p}= 0$. This gives the desired result.

(2) Since the map $f^*$ continuous, this assertion is immediate.  

(3) It is just a special case of (2).
\end{proof}

\begin{rem}
Here $\T(S)$ is not necessarily same with $\T(R)\otimes_R S$, since in general, $R\rightarrow S$ is not flat. So we cannot compare them. But, we can compare their supports, see results in the next section.
\end{rem}

Now, we state the main result of this section.

\begin{thm}
Consider the ring extension $R\rightarrow S$ of Noetherian rings satisfying (\textbf{*}) above. For Lyubeznik functor $\T$, $\Supp_{R} \T(R)$ is a Zariski closed subset in $\Spec R$ if and only if $\Supp_{S} (\T(R)\otimes_R S)$ is a Zariski closed subset in $\Spec S$.
\end{thm}

\begin{proof}
The `only if' part of the assertion is immediate from (2) of above Proposition 4.1. 

For `if' part, proof is almost similar to that of Theorem 3.4 and for sake of completeness we repeat it: From hypothesis, using Lemma 2.3 we get $\min_S (\T(R)\otimes_R S)$ is finite. From Lemma 3.1 we get $\min_R (\T(R)\otimes_R S)$ is also finite. 
Since $R$ is a direct summand of $S$, $\T(R)$ is a direct summand of $\T(S)$, i.e. $\T(R)$ is an $R$-submodule of $\T(S)$. So, we can apply Proposition 3.3 and we find that $\min_R \T(R)\subset\min_R (\T(R)\otimes_R S)$ and latter is finite. Thus $\min_R \T(R)$ is also finite and thus again by Lemma 2.3, $\Supp_{R} \T(R)$ is also a Zariski closed subset in $\Spec R$.  
\end{proof}

\section{comparison of $\Supp_S(\T(R)\otimes_R S)$ and $\Supp_S \T(S)$}

From Lemma 3.1 of \cite{Ly1}, we know that for flat and faithfully flat ring extension $R\rightarrow S$, for Lyubeznik functor $\T$, $\T(R)\otimes_R S= \T(S)$ and as a consequence $\Supp_S(\T(R)\otimes_R S)= \Supp_S \T(S)$. But, even for a pure extension $R\rightarrow S$, neither we can have a  nice relation between $\T(R)\otimes_R S$ and $\T(S)$, nor we can compare the sets $\Supp_S(\T(R)\otimes_R S)$ and $\Supp_S \T(S)$. But, when the extension $R\rightarrow S$ satisfies (\textbf{*}) (for its definition, see previous section), at least in this situation, we can compare $\Supp_S(\T(R)\otimes_R S)$ and $\Supp_S \T(S)$. 

We observe the containment $\Supp_S(\T(R)\otimes_R S)$ in $\Supp_S \T(S)$, when the ring extension $R\rightarrow S$ satisfies (\textbf{*}) and we present the result in the following proposition.

\begin{prop}
Consider the ring extension $R\rightarrow S$ of Noetherian rings satisfying (\textbf{*}) above. Then, for Lyubeznik functor $\T$, $\Supp_S(\T(R)\otimes_R S)\subset \Supp_S \T(S)$.
\end{prop}

\begin{proof}
Let $R\stackrel{f}{\rightarrow}S$ is a ring homomorphism such that $S=R\oplus T$. Consider the natural map $\Spec S\stackrel{f^*}{\rightarrow}\Spec R$. For $\q\in \Spec S$, set $\p= \q\cap R= f^*(\q)$. Let $\q\in \Supp_S(\T(R)\otimes_R S)$. From first assertion of Proposition 19 of Section II.4.4 of \cite{Bo}, we find that $\q\in {f^*}^{-1}(\Supp_R(\T(R)))$ and equivalently $\p= f^*(\q)\in \Supp_R \T(R)$. Now, from Remark 6 above, we find that $R= S/T$ as well and for ideal $I\subset R$, image of $IS$ in $R= S/T$ is $I$ again. Thus $\T(R)$ can be thought as an $S$-module. Since $S= R\oplus T$, we have the following relation between the $S$-modules: $\T(S)= \T(R)\oplus W$. Tensoring with $S_{\q}$ over $S$ we get, $(\T(S))_{\q}= \T(S_{\q})= \T(S)\otimes_S S_{\q}= (\T(R)\otimes_S S_{\q})\oplus (W\otimes_S S_{\q})= \T(R\otimes_S S_{\q})\oplus (W\otimes_S S_{\q})= \T(R_{\p})\oplus W_{\q}=(\T(R))_{\p}\oplus W_{\q}$ (see Remark 6, above). Thus $\T(R))_{\p}\neq 0$ implies that $(\T(S))_{\q}\neq 0$. Thus we conclude. 
\end{proof}

Lastly, we observe that when ring extension $R\rightarrow S$ satisfying (\textbf{*}), Zariski closedness of $\Supp_S \T(S)$ implies Zariski closedness of $\Supp_S(\T(R)\otimes_R S)$. 

\begin{prop}
Consider the ring extension $R\rightarrow S$ satisfying (\textbf{*}). Then for Lyubeznik functor $\T$, if $\Supp_S \T(S)$ is Zariski closed subset in $\Spec S$, then $\Supp_S(\T(R)\otimes_R S)$ is also Zariski closed subset in $\Spec S$. 
\end{prop}

\begin{proof}
Here, $R$ is a direct summand of $S$, so $\T(R)\subset \T(S)$. This further implies that $\min_R \T(R)\subset \min_R \T(S)$ (see (1) of Lemma 3.1). Now $\min_R \T(S)\subset \{q\cap R: q\in\min_S \T(S)\}$ (see (2) of Lemma 3.1). From hypothesis, we get $\min_S \T(S)$ is finite and so is $\min_R \T(R)$. Thus we get $\Supp_R \T(R)$ is closed in $\Spec R$ and from Proposition 4.1 above, we get $\Supp_S(\T(R)\otimes_R S)$ is also closed in $\Spec S$.
\end{proof}

\ 

{\textbf{Acknowledgement:}}\newline

I would like to thank Tony J. Puthenpurakal for his invaluable comments and suggestions.


\end{document}